\documentclass[11pt,a4paper]{article}

\usepackage{a4wide,amssymb,graphics,graphicx,textcomp}
\usepackage{enumerate,textcomp,multirow,amsmath}
\usepackage{algorithm}
\usepackage{algorithmic}
\usepackage{url}

\newtheorem{theorem}{Theorem}[section]
\newtheorem{lemma}[theorem]{Lemma}
\newtheorem{remark}[theorem]{Remark}

\newtheorem{example}[theorem]{Example}

\numberwithin{equation}{section}
\numberwithin{table}{section}
\numberwithin{figure}{section}

\newcommand {\mat}  [1] {\left[\begin{array}{#1}}
\newcommand {\rix}      {\end{array}\right]}

\catcode`@=11     
\font\tenex=cmex10 
\newdimen\p@renwd
\setbox0=\hbox{\tenex B} \p@renwd=\wd0 
\def\bmat#1{\begingroup \m@th
  \setbox\z@\vbox{\def\cr{\crcr\noalign{\kern2\p@\global\let\cr\endline}}%
    \ialign{$##$\hfil\kern2\p@\kern\p@renwd&\thinspace\hfil$##$\hfil
      &&\quad\hfil$##$\hfil\crcr
      \omit\strut\hfil\crcr\noalign{\kern-\baselineskip}%
      #1\crcr\omit\strut\cr}}%
  \setbox\tw@\vbox{\unvcopy\z@\global\setbox\@ne\lastbox}%
  \setbox\tw@\hbox{\unhbox\@ne\unskip\global\setbox\@ne\lastbox}%
  \setbox\tw@\hbox{$\kern\wd\@ne\kern-\p@renwd\left[\kern-\wd\@ne
    \global\setbox\@ne\vbox{\box\@ne\kern2\p@}%
    \vcenter{\kern-\ht\@ne\unvbox\z@\kern-\baselineskip}\,\right]$}%
  \null\;\vbox{\kern\ht\@ne\box\tw@}\endgroup}
\catcode`@=12    

\newcommand{\backmatter}{\appendix
\def\chaptermark##1{\markboth{%
\ifnum  \c@secnumdepth > \m@ne  \@chapapp\ \thechapter:  \fi  ##1}{%
\ifnum  \c@secnumdepth > \m@ne  \@chapapp\ \thechapter:  \fi  ##1}}%
\def\sectionmark##1{\relax}}

\makeatletter
\newcommand*{\rom}[1]{\expandafter\@slowromancap\romannumeral #1@}
\makeatother

\def\real{\mathop{\mathrm{Re}}}
\def\imag{\mathop{\mathrm{Im}}}

\newif\ifMDlatex
\catcode`@=11     
\def\MD@us#1{\csname#1style\endcsname}
\def\MD@uf#1{\csname#1font\endcsname}
\def\MD@t{text}
\def\MD@s{script}
\def\MD@ss{scriptscript}
\newdimen\MD@unit
\MD@unit\p@
\def\MD@changestyle#1{
  \relax\MD@unit0.1\fontdimen6\MD@uf{#1}0
  \everymath\expandafter{\the\everymath\MD@us{#1}}
}
\def\MD@dot{$\m@th\ldotp$}
\def\MD@palette#1{\mathchoice{#1\MD@t}{#1\MD@t}{#1\MD@s}{#1\MD@ss}}
\def\MD@ddots#1{{\MD@changestyle{#1}%
  \mkern1mu\raise7\MD@unit\vbox{\kern7\MD@unit\hbox{\MD@dot}}%
  \mkern2mu\raise4\MD@unit\hbox{\MD@dot}%
  \mkern2mu\raise \MD@unit\hbox{\MD@dot}\mkern1mu}}%
\def\MD@iddots#1{{\MD@changestyle{#1}%
  \mkern1mu\raise \MD@unit\hbox{\MD@dot}%
  \mkern2mu\raise4\MD@unit\hbox{\MD@dot}%
  \mkern2mu\raise7\MD@unit\vbox{\kern7\MD@unit\hbox{\MD@dot}}}}%
\def\MD@vdots#1{\vbox{\MD@changestyle{#1}%
    \baselineskip4\MD@unit\lineskiplimit\z@
    \kern6\MD@unit\hbox{\MD@dot}\hbox{\MD@dot}\hbox{\MD@dot}}}%
\ifMDlatex
  \DeclareRobustCommand\ddots{\mathinner{\MD@palette\MD@ddots}}%
  \DeclareRobustCommand\iddots{\mathinner{\MD@palette\MD@iddots}}%
  \DeclareRobustCommand\vdots{\mathinner{\MD@palette\MD@vdots}}%
\else
  \def\ddots{\mathinner{\MD@palette\MD@ddots}}%
  \def\iddots{\mathinner{\MD@palette\MD@iddots}}%
  \def\vdots{\mathinner{\MD@palette\MD@vdots}}%
\fi
\catcode`@=12    

\newcommand {\comment}[1]{} 
\newcommand{\C}{{\mathbb C}}
\newcommand{\R}{{\mathbb R}}

\begin{document}
\title{\hspace{1.3cm} Optimizing Rayleigh quotient with symmetric constraints and their applications to perturbations of the structured polynomial eigenvalue problem}
\author{  Anshul Prajapati $^*$ \qquad Punit Sharma\thanks{A.P. acknowledges the support of the CSIR Ph.D. grant by Ministry of Science \& Technology, Government of India. P.S. acknowledges the support of the DST-Inspire Faculty Award (MI01807-G) by Government of India and FIRP project (FIRP/Proposal Id - 135) by IIT Delhi, India.
 Email: \{maz198078, punit.sharma\}@maths.iitd.ac.in.} \\ 
Department of Mathematics\\
Indian Institute of Technology Delhi\\
Hauz Khas, 110016, India
}

\date{}

\maketitle

\begin{abstract}
For a Hermitian matrix $H \in \C^{n,n}$ and symmetric matrices $S_0,
S_1,\ldots,S_k \in \C^{n,n}$, we consider the problem of computing the supremum of
$\left\{ \frac{v^*Hv}{v^*v}:~v\in \C^{n}\setminus \{0\},\,v^TS_iv=0~\text{for}~i=0,\ldots,k\right\}$. For this, we derive an estimation in the form of minimizing the second largest eigenvalue of a parameter depending Hermitian matrix, which is exact when the eigenvalue at the optimal is simple. The results are then applied to compute the eigenvalue backward errors of higher degree matrix polynomials with T-palindromic, T-antipalindromic, T-even, T-odd, and skew-symmetric structures. The results are illustrated by numerical experiments.
\end{abstract}

{\textbf{ keyword}}
structured matrix polynomial, spectral value sets, $\mu$-values, perturbation theory, eigenvalue backward error, nearness problems for matrix polynomials

\section{Introduction}

Let $H \in \C^{n,n}$ be Hermitian  and $S_0,S_1,\ldots,S_k \in \C^{n,n}$ be symmetric matrices. In this paper, we maximize the Rayleigh quotient of $H$ with respect to certain constraints involving symmetric matrices $S_0,S_1,\ldots,S_k$. More precisely, we compute
\begin{equation}\label{eq:cenprob}
m_{hs_0s_1\ldots s_k}(H,S_0,S_1,\ldots,S_k):=\sup\left\{\frac{v^*Hv}{v^*v} :~v\in \C^{n} \setminus \{0\},\,v^TS_iv=0~\text{for}~i=0,\ldots,k 
\right\}, \tag{$\mathcal{P}$}
\end{equation}
where $T$ and $*$ denote respectively the transpose and the conjugate transpose of a matrix or a vector. 

Such problems occur in stability analysis of uncertain systems and in the eigenvalue perturbation theory of matrices and matrix polynomials~\cite{Doy82,Kar03,PacD93,ZhoDG96}. The $\mu$-value of $M \in \C^{n,n}$ with respect to perturbations from the structured class $\mathcal S \subseteq C^{n,n}$ is denoted by $\mu_\mathcal S (M)$ and defined as 
\begin{equation*}
\mu_\mathcal S(M):=\left ( \inf \left \{ \|\Delta\|:~\Delta \in \mathcal S,\,\text{det}(I_n-\Delta M)=0 \right\} \right)^{-1},
\end{equation*}
where $I_n$ is the identity matrix of size $n\times n$ and $\|.\|$ is the matrix spectral norm. A particular case of problem~\eqref{eq:cenprob} with only one symmetric constraint (i.e., when $k=0$) is used to characterize 
the $\mu$-value of the matrix under skew-symmetric perturbations~\cite{Kar11}. Indeed, when $\mathcal S$ is the set of all skew-symmetric matrices of size $n \times n$, then M. Karow in~\cite{Kar11} shown that 
\[
\mu_\mathcal S(M)=(m_{hs_0}(H,S_0))^{\frac{1}{2}},~\text{where}~
H=M^*M~\text{and}~S_0=M+M^T.
\]
Moreover, an explicit computable formula was obtained for $m_{hs_0}(H,S_0)$ in~\cite[Theorem 6.3]{Kar11} and given by
\begin{equation}\label{eq:mikaresult}
m_{hs_0}(H,S_0)=\inf_{t\in[0,\infty)} \lambda_2\left(\begin{bmatrix}
H & t\overline S_0 \\ t S_0 & \overline H
\end{bmatrix}\right),
\end{equation}
where $\lambda_2(A)$ stands for the second largest eigenvalue of a Hermitian matrix $A$. However, we note that problem~\eqref{eq:cenprob}
with more than one symmetric constraints is unknown yet and still open. 

Another motivation to study problem~\eqref{eq:cenprob} comes from the eigenvalue backward error computation of structured matrix polynomials. Let $P(z)=\sum_{j=0}^m z^j A_j$ be a regular matrix polynomial with $A_0,A_1,\ldots, A_m \in \C^{n,n}$. For a given $\lambda \in \C$, the smallest perturbation $(\Delta_0,\ldots,\Delta_m)$ from some perturbation set $\mathbb S \subseteq (\C^{n,n})^{m+1}$ such that $\text{det}\left(\sum_{j=0}^m \lambda^j (A_j-\Delta_j)\right)=0$, is called the structured eigenvalue backward error of $\lambda$ as an approximate eigenvalue of $P(z)$. The matrix polynomials in most engineering applications follow some symmetry structure~\cite{TisM01,Lan66,LanR95}. The perturbation set $\mathbb S$ refers to a symmetry structure in the coefficients of the matrix polynomial. For example, if $A_j^T=A_{m-j}$ ($A_j^T=-A_{m-j}$), for $j=0,\ldots,m$, then $P(z)$ is called T-palindromic (T-antipalindromic), and if $A_j^T=-A_j$ for $j=0,\ldots, m$, then it is called a skew-symmetric matrix polynomial. Similarly, $P(z)$ is called T-even (T-odd) if $A_j^T=(-1)^jA_j$  $\left(A_j^T=(-1)^{j+1}A_j \right)$ for $j=0,\ldots,m$. If the conjugate transpose $*$ replaces the transpose T, then the corresponding matrix polynomials called $*$-palindromic, $*$-antipalindromic, skew-Hermtian, $*$-even, and $*$-odd, see~\cite{TisM01,MacMMM06b,Zag02,Meh91} and references therein. 

The problem of finding structured eigenvalue backward error has been studied, and computable formulas have been obtained for polynomials with Hermitian and related structures in~\cite{BorKMS14}, and for $*$-palindromic polynomials and T-palindromic pencils in~\cite{BorKMS15}.

It was noticed in~\cite{Sha16} that the structured eigenvalue backward error problem for T-palindromic, T-antipalindromic, skew-symmetric, T-even, and T-odd polynomials reduce to the problem~\eqref{eq:cenprob} where the number of symmetric constraints in~\eqref{eq:cenprob} depends on the degree of the polynomial $P(z)$. For the pencil case (i.e., when $m=1$), the reduced problem~\eqref{eq:cenprob} involves only one symmetric constraint, and thus the result~\cite[Theorem 6.3]{Kar11} obtained the eigenvalue backward errors for these structures~\cite{BorKMS15,Sha16}. However, for higher degree polynomials (i.e., when $m>1$), the eigenvalue backward error is not known for these structures since the reduced problem~\eqref{eq:cenprob} involves more than one symmetric constraint. In such cases, the computation of the backward error may include obtaining an appropriate extension of~\cite[Theorem 6.3]{Kar11}. This gives another motivation to study the problem~\eqref{eq:cenprob} .

We note that an upper bound to the problem~\eqref{eq:cenprob} with more than one symmetric constraint were obtained in~\cite{Sha16}. Motivated by~\cite{Kar11}, we aim to derive an explicit computable formula 
for~\eqref{eq:cenprob}. We found that our results improve the bounds given in~\cite{Sha16}. 

\subsection{Contribution and outline of the paper}
In Section~\ref{sec:Prelim}, we introduce some notation and terminologies. We also recall some preliminary results from the literature that will be used in later sections. 

In Section~\ref{sec:formulation}, we derive an upper bound for~\eqref{eq:cenprob} in terms of minimizing the second largest eigenvalue of a parameter depending Hermitian matrix. An approximation to this upper bound can be easily computed in Matlab. The upper bound is shown to be equal to the exact value of the supremum in~\eqref{eq:cenprob} if the eigenvalue at the optimal is simple.

The reduction of the structured eigenvalue backward error problem to a problem of the form~\eqref{eq:cenprob} was obtained in~\cite{BorKMS15,Sha16} for matrix polynomial with structures like T-palindromic, T-antipalindromic, T-even, T-odd, and skew-symmetric polynomials. We restate these results in Section~\ref{sec:sber} in our notation. The results from Section~\ref{sec:formulation} are then applied to derive computable formulas for the structured eigenvalue backward errors of matrix polynomials under consideration. 

In Section~\ref{sec:numerical}, we present numerical experiments to highlight that our results give better estimation to the supremum in~\eqref{eq:cenprob} than the one obtained in~\cite{Sha16}. The significance of structure-preserving and arbitrary perturbations on the eigenvalue backward errors is also illustrated in Section~\ref{sec:numerical}.

\section{Preliminaries}\label{sec:Prelim}

In the following, we denote the spectral norm of a matrix or a vector by $\|\cdot\|$, the smallest and the largest eigenvalues of a Hermitian matrix $A$ respectively by $\lambda_{\min}(A)$ and $\lambda_{\max}(A)$, the Moore-Penrose pseudoinverse of a matrix or a vector $X$ by $X^\dagger$, $\otimes$ denotes the Kronecker product, and the smallest singular value of a matrix $A$ by $\sigma_{\min}(A)$. We use
$\text{Herm}(n)$ and $\text{Sym}(n)$  to denote the set of Hermitian  and symmetric matrices of size $n \times n$. The second-largest eigenvalue of a Hermitian matrix $A$ is denoted by $\lambda_2(A)$, and the second-largest singular value of a matrix $A$ is denoted by $\sigma_2(A)$.  The symbol $i$ denotes the imaginary unit and satisfying the equation $i^2 = -1$.

For $H\in {\rm Herm}(n)\; \text{and}\; S_0,S_1,\ldots,S_k\in {\rm Sym}(n)$, we define
\begin{equation}\label{eq:mutild}
\widetilde{m}_{hs_0{s_1}\ldots{s_k}}(H,S_0,\ldots,S_k):=\inf\left \{v^*Hv\; | \; v \in \C^{n}, \|v\|=1,\, v^TS_iv=0~\text{for}~i=0,\ldots,k \right \}.
\end{equation}

Although our main focus would be the problem~\eqref{eq:cenprob}, we note that the corresponding results for~\eqref{eq:mutild} follow directly from the results for~\eqref{eq:cenprob}, since
\[
\tilde{m}_{hs_0{s_1}\ldots{s_k}}(H,S_0,\ldots,S_k)=-m_{hs_0{s_1}\ldots{s_k}}(-H,-S_0,\ldots,-S_k).
\]

As mentioned earlier, the problem~\eqref{eq:cenprob} with only one symmetric constraint is solvable, and a computable formula for $m_{hs_0}(H,S_0)$ is known by~\cite{Kar11}. We recall this result in the following theorem. 

\begin{theorem}{\rm \cite[Theorem 6.3]{Kar11}}\label{thm:mikaresult}
Let $H \in {\rm Herm}(n)$ and $S\in {\rm Sym}(n)$. Then 
\[
\sup\left \{v^*Hv:\;v\in \C^n,\;v^TSv=0,\; \|v\|=1 \right \}=\inf_{t \in [0,\infty)}
\lambda_2\left(\begin{bmatrix}
H & t\overline{S}\\tS & \overline{H}
\end{bmatrix}  \right).
\]
Moreover, if ${\rm rank}(S) \geq 2$, then the infimum is attained in the interval $[0,t_1]$, where $t_1=\frac{2\|H\|}{\sigma_2(S)}$.
\end{theorem}

The following lemma provides helpful information in analysing the extrema of eigenvalue functions of a real parameter.

\begin{lemma}{\rm \cite{Bau85}}\label{dlemma}
Let $G, H \in {\rm Herm}(n)$  and let the map $L: \R \rightarrow \R$ be given by $L(t):=\lambda_k(G+tH)$, where $\lambda_k(G+tH)$ is the $k$th largest eigenvalue of $G+tH$.  Let the columns of the matrix $U\in \C^{n, n}$ form an orthonormal basis of the eigenspace of the eigenvalue $\lambda_k(G)$ of $G$. Then the left and right directional derivatives of $L$ in $t=0$ exists and we have,
\begin{eqnarray*}
 \begin{split}
        \frac{d}{dt}L(0)_+ & := \lim_{\epsilon \rightarrow 0\; \epsilon >0} \frac{\lambda_k(G+\epsilon H)-\lambda_k(G)}{\epsilon} = \lambda_k(U^*HU); \\
        \frac{d}{dt}L(0)_- & := \lim_{\epsilon \rightarrow 0\; \epsilon >0} \frac{\lambda_k(G-\epsilon H)-\lambda_k(G)}{-\epsilon} = \lambda_{n-k+1}(U^*HU).
    \end{split}
\end{eqnarray*}

If in particular, $m=1$, then $L$ is differentiable at $t=0$, $u:=U \in \C^n\setminus\{0\}$ and 
\[
\frac{d}{dt}L(0) = \lambda_k(U^*HU)=u^*Hu.
\]
\end{lemma}

The following two lemmas that give lower bounds on the sum of the {\rm rank} of two matrices will be used in computing problems~\eqref{eq:cenprob} with more than one symmetric constraints. 

\begin{lemma}\label{lem:rank1}
Let $A$ and $B$ be two matrices of the same size and let their row spaces be $R_1$ and $R_2$, and their column spaces be $C_1$ and $C_2$, respectively. Then
\[
{\rm rank}(A) + {\rm rank}(B) -{\rm dim}(R_1\cap R_2)-{\rm dim}(C_1\cap C_2)\leq {\rm rank}(A+B).
\]
\end{lemma}

\begin{lemma}\label{changecondition}
Let $S'$ and  $S''$ be two square matrices of size $n$ with ${\rm rank}(S')\geq 2$ and ${\rm rank}(S'')\geq 2$. If ${\rm rank}(\begin{bmatrix}S'&S'' \end{bmatrix} + {\rm rank}\Bigg(\begin{bmatrix}
S'\\S''
\end{bmatrix}\Bigg)\geq {\rm {\rm rank}}(S') + {\rm {\rm rank}}(S'') + 2 $, then ${\rm {\rm rank}}(S'+S'')\geq 2$. 
\end{lemma}
\begin{proof}
For matrices $S'$ and $S''$, we have from Lemma~\ref{lem:rank1} that
\begin{equation}\label{lemrank:1}
    {\rm rank}(S') + {\rm rank}(S'') -{\rm dim}(R_1\cap R_2)-{\rm dim}(C_1\cap C_2)\leq {\rm rank}(S'+S''),
\end{equation}
where $R_1$ and  $R_2$ are the row spaces and $C_1$ and $C_2$ are the column spaces of $S'$ and $S''$, respectively.
Also, we have 
\begin{eqnarray}\label{lemrank:2}
        {\rm rank}\left(\begin{bmatrix} S' \\S'' 
        \end{bmatrix}\right) = {\rm dim}(R_1+R_2) 
        & =& {\rm dim}(R_1)+{\rm dim}(R_2)-{\rm dim}(R_1\cap R_2)\nonumber \\
        & =& {\rm rank}(S')+{\rm rank}(S'')-{\rm dim}(R_1\cap R_2),
\end{eqnarray}
and 
\begin{eqnarray}\label{lemrank:3}
{\rm rank}(\begin{bmatrix}S' & S''\end{bmatrix}) = {\rm dim}(C_1+C_2) & = {\rm dim}(C_1)+{\rm dim}(C_2)-{\rm dim}(C_1\cap C_2) \nonumber\\
        & = {\rm rank}(S')+{\rm rank}(S'')-{\rm dim}(C_1\cap C_2).
\end{eqnarray}

From~\eqref{lemrank:1}, \eqref{lemrank:2} and~\eqref{lemrank:3}, we get
\begin{eqnarray}\label{lemrank:4}
{\rm rank}\left(\begin{bmatrix} S' \\S'' \end{bmatrix}\right) +
{\rm rank}(\begin{bmatrix}S' & S''\end{bmatrix}) \leq {\rm rank}(S')+{\rm rank}(S'')+{\rm rank} (S'+S'').
\end{eqnarray}
Now using the assumption ${\rm rank}(\begin{bmatrix}S'&S'' \end{bmatrix} + {\rm rank}\Bigg(\begin{bmatrix}
S'\\S''
\end{bmatrix}\Bigg)\geq {\rm {\rm rank}}(S') + {\rm {\rm rank}}(S'') + 2 $ in 
~\eqref{lemrank:4} yields that ${\rm {\rm rank}}(S'+S'')\geq 2$. 
\end{proof}

\section{Maximizing Rayleigh quotient with symmetric constraints}
\label{sec:formulation}

Motivated by~\cite{Kar11}, we obtain a computable formula to estimate the supremum in~\eqref{eq:cenprob} with more than one symmetric constraint. 
For this, let us fix the following terminology that will be repeatedly used throughout the paper. 
For  $H \in {\rm Herm}(n)$, $S_0, \ldots, S_k\in {\rm Sym}(n)$, and
$(t_0,\ldots,t_{2k})\in \R^{2k+1}$, we define 
\begin{equation}\label{eq:defmat1}
f(t_0,\ldots,t_{2k}):=(t_0+it_1)S_0+\cdots+(t_{2k-2}+it_{2k-1})S_{k-1}+t_{2k}S_k,
\end{equation}
\begin{equation}\label{eq:defmat2}
G:=\mat{cc}H & 0\\ 0 &\overline H \rix,\quad H_{2k}:=\mat{cc} 0 & \overline S_k \\ S_k & 0 \rix,
\end{equation}
and for $j=0,\ldots, k-1$, we define 
\begin{eqnarray}\label{eq:defmat3}
H_{2j}:=\mat{cc} 0 & \overline S_j \\ S_j & 0 \rix, \quad \text{and}\quad H_{2j+1}:=\mat{cc} 0 & \overline {iS_j} \\ iS_j & 0 \rix.
\end{eqnarray}

In this section, we will find out that the computation of 
$m_{hs_0{s_1}\ldots{s_k}}(H,S_0,\ldots,S_k)$  will lead to a minimization problem of a function of the following form
\[
\psi:\R^{2k+1} \longrightarrow \R,\quad (t_0,\ldots,t_{2k}) \longmapsto \lambda_2\left ( G+t_0H_0+\cdots+t_{2k} H_{2k}\right ).
\]
The following lemma provides helpful information to analyse the extrema of $\psi$, the proof of which is inspired by~\cite[Theorem 3.2]{BorKMS14}.

\begin{lemma}\label{lem:existence}
Let $H \in {\rm Herm}(n)$ and $S_0, \ldots, S_k\in {\rm Sym}(n)$. 
Assume that for any $(t_0,\ldots,t_{2k}) \in \R^{2k+1}\setminus \{0\}$, we have
${\rm rank}(f(t_0,\ldots,t_{2k}))\geq 2$, where $f(t_0,\ldots,t_{2k})$ is defined by~\eqref{eq:defmat1}. Then the function $\psi : \R^{2k+1}\rightarrow \R$, $(t_0,\ldots,t_{2k}) \longmapsto \lambda_2(F(t_0,\ldots,t_{2k}))$, where 
\begin{equation}\label{eq:defF}
F(t_0,\ldots,t_{2k})=\begin{bmatrix}
H & \overline{f(t_0,\ldots,t_{2k})}\\
f(t_0,\ldots,t_{2k}) & \overline{H}
\end{bmatrix}
\end{equation}
has a global minimum
\[
\widehat \lambda_2=\min_{(t_0,\ldots,t_{2k})\in \R^{2k+1}} \psi(t_0,\ldots,t_{2k})
\]
 in the region $t_0^2+\cdots+t_{2k}^2 \leq \beta^2$ where $\beta=\frac{\lambda_{\max}(H)-\lambda_{\min}(H)}{c}$ and 
\[
c = \min \left \{\lambda_2(t_0H_0+\cdots+t_{2k}H_{2k})  :~(t_0,\ldots,t_{2k}) \in \R^{2k+1},\, t_0^2+\cdots+t_{2k}^2=1\right\},
\]
with matrices $H_0,\ldots, H_{2k}$ defined by~\eqref{eq:defmat2}-\eqref{eq:defmat3}.
\end{lemma}
\begin{proof}
To prove that $\psi$ has a global minimum, we will show that there exists a constant $\beta>0$ such that for all $(t_0,\ldots,t_{2k}) \in \R^{2k+1}$ with $t_0^2+\cdots + t_{2k}^2 > \beta^2$, we have 
\begin{equation}\label{eq:lemexistprof1}
\psi(t_0,\ldots,t_{2k})\geq \psi(0,\ldots,0).
\end{equation}
Since the closed ball $\mathcal N_\beta=\left\{ (t_0,\ldots,t_{2k}) \in \R^{2k+1}:~t_0^2+\cdots + t_{2k}^2 \leq \beta^2\right\}$ is compact in $\R^{2k+1}$ and since the function $\psi$ is continuous as eigenvalues depend continuously on the entries of a matrix, $\psi$ will have a global minimum $\widehat \lambda_2 \leq \psi(0,\ldots,0)$ on $\mathcal N_\beta$. This from~\eqref{eq:lemexistprof1} will then imply that $\widehat \lambda_2 \leq \psi(t_0,\ldots,t_{2k})$ for all $(t_0,\ldots,t_{2k})\in \R^{2k+1}$, i.e., $\widehat \lambda_2$ is the global minimum of $\psi$.

Thus, let 
\[
c = \inf \left \{\lambda_2(t_0H_0+\cdots+t_{2k}H_{2k})  :~(t_0,\ldots,t_{2k}) \in \R^{2k+1},\, t_0^2+\cdots+t_{2k}^2=1\right\}.
\]
Then $c \geq 0$. Indeed, if we let 
\[
\widetilde f(t_0,\ldots,t_{2k})=t_0H_0+\cdots + t_{2k}H_{2k},
\]
and $T=\text{diag}(-1,1)\otimes I_n$, then
\[
\widetilde f(-t_0,\ldots,-t_{2k})=T\widetilde f(t_0,\ldots,t_{2k})T^{-1}.
\]
This implies that $\lambda$ is an eigenvalue of $\widetilde f(t_0,\ldots,t_{2k})$ if and only if $-\lambda$ is an eigenvalue of $\widetilde f(t_0,\ldots,t_{2k})$. Also the assumption $\text{rank}( f(t_0,\ldots,t_{2k}))\geq 2$ implies that $\text{rank}(\widetilde f(t_0,\ldots,t_{2k}))\geq 4$ for all $(t_0,\ldots,t_{2k})\in \R^{2k+1}\setminus \{0\}$. This implies that $\lambda_2(\widetilde f(t_0,\ldots,t_{2k})) > 0$ for all $(t_0,\ldots,t_{2k}) \in \R^{2k+1}$ with $t_0^2+\cdots+t_{2k}^2=1$.
Moreover, since $\lambda_2(\widetilde f(t_0,\ldots,t_{2k})) $ is a continuous function as eigenvalues depend continuously on the entires of the matrix, the infimum $c$ is attained on unit sphere in $\R^{2k+1}$. This implies that $c >0$ because $\lambda_2(\widetilde f(t_0,\ldots,t_{2k})) $ always take positive values on the unit sphere. 

Now define $\beta=\frac{\lambda_{\max}(H)-\lambda_{\min}(H)}{c}$ and  let 
$(t_0,\ldots,t_{2k}) $ be such that $t_0^2+\cdots+t_{2k}^2 > \beta^2$. Then there exists $r > \beta$ so that $t_0^2+\cdots+t_{2k}^2=r^2 > \beta^2$. By using Wehl's theorem~\cite{HorJ85}, for any two Hermitian matrices
$A,B\in \C^{n,n}$ we have that 
\[
\lambda_2(A+B) \geq \lambda_2(A)+\lambda_{\min}(B).
\]
This yields that 
\begin{eqnarray*}
\psi(t_0,\ldots,t_{2k})&=&\lambda_2(F(t_0,\ldots,t_{2k}))\\
&=& \lambda_2(G+t_0H_0+\cdots+t_{2k}H_{2k})\quad \quad(\because \text{from}~\eqref{eq:defmat2}-\eqref{eq:defmat3})\\
&=& \lambda_2(t_0H_0+\cdots+t_{2k}H_{2k}) + \lambda_{\min}(G)\\
&=& r \lambda_2(\frac{t_0}{r}H_1+\cdots+\frac{t_{2k}}{r}H_{2k})+ \lambda_{\min}(G)\\
&\geq & rc+ \lambda_{\min}(G)\\
&\geq & \beta c +  \lambda_{\min}(G)\quad \quad (\because r \geq \beta)\\
&=& \psi(0,\ldots,0), 
\end{eqnarray*}
where the last identity follows using the facts that  $\psi(0,\ldots,0)=\lambda_2(G)=\lambda_{\max}(H)=\beta c+\lambda_{\min}(G)$ and $\lambda_{\min}(G)=\lambda_{\min}(H)$. This completes the proof.
\end{proof}

\begin{remark}{\rm 
Although the rank condition in Lemma~\ref{lem:existence} seems to be strong, it was noticed that the symmetric matrices arising in most practical applications satisfy the rank condition.  For example, we show  in Section~\ref{sec:applications} that the symmetric matrices occurring while finding the eigenvalue backward errors for T-palindromic, T-antipalindromic, T-even, and  T-odd  matrix polynomials satisfy the rank condition in all the cases.
}
\end{remark}

 Our next result is a generalization of~\cite[Lemma 6.1]{Kar11} which provides an upper bound to $m_{hs_0,\ldots,s_k}$.

 \begin{lemma}\label{lem:bound}
 Let the function $\psi:\R^{2k+1}\longmapsto \R$ be as defined in Lemma~\ref{lem:existence} for matrices $H \in {\rm Herm}(n)$ and $S_0,  \ldots, S_k\in {\rm Sym}(n)$. Then 
 \begin{equation*}
 m_{hs_0\ldots s_k}(H,S_0,\ldots,S_k)  \leq \inf_{t_0,\ldots,t_{2k}\in\R}\psi(t_0,\ldots,t_{2k}).
 \end{equation*}
 \end{lemma}
\begin{proof}
For a unit vector $v\in \C^n$, define $\gamma_v := \Bigg\{ \begin{bmatrix}
z_1v\\ \overline{z_2v}
\end{bmatrix}:~ z_1,z_2 \in \C\Bigg\}.$ Then $\gamma_v$ is a 2-dimensional subspace of $\C^{2n}$. For any vector $x=\begin{bmatrix}
z_1v\\ \overline{z_2v} \end{bmatrix}\in\gamma_v$ and $(t_0,\ldots,t_{2k}) \in \R^{2k+1}$, we have
\begin{eqnarray}\label{eq:lbound1}
x^*F(t_0,\ldots,t_{2k})x
& =& \begin{bmatrix}
z_1v\\ \overline{z_2v}
\end{bmatrix}^*\begin{bmatrix}H & \overline{f(t_0,\ldots,t_{2k})}\\ f(t_0,\ldots,t_{2k}) & \overline{H}\end{bmatrix}\begin{bmatrix}
z_1v\\ \overline{z_2v}
\end{bmatrix} \nonumber \\
&=& \begin{bmatrix}
z_1v\\ \overline{z_2v}
\end{bmatrix}^*\begin{bmatrix}
z_1Hv+\overline{z_2 f(t_0,\ldots,t_{2k})v}\\ \overline{z_2Hv}+z_1{f(t_0,\ldots,t_{2k})v}
\end{bmatrix}\nonumber \\
&=& (|z_1|^2+|z_2|^2)v^*Hv+2\real{\left(z_1z_2(v^Tf(t_0,\ldots,t_{2k})v)\right)},
\end{eqnarray}
where $f(t_0,\ldots,t_{2k})$ and  $F(t_0,\ldots,t_{2k})$ are respectively defined by~\eqref{eq:defmat1} and~\eqref{eq:defF}.
If we suppose that $v^TS_jv=0$ for each $j=0,\ldots,k$, then for any 
$x \in \gamma_v$ such that $\|x\|=1$, we have from~\eqref{eq:lbound1} that
\[
x^*F(t_0,\ldots,t_{2k})x=v^*Hv,
\]
since $v^Tf(t_0,\ldots,t_{2k})v=0$ and $\|x\|=1$ implies that $|z_1|^2+|z_2|^2=1$. Thus by Courant-Fischer max-min principle~\cite{HorJ85}, we have
\begin{equation*}
    \begin{split}
 \psi(t_0,\ldots,t_{2k})~= ~\lambda_2(F(t_0,\ldots,t_{2k})) ~\geq \min_{x\in \gamma_v,\, \|x\|=1}x^*F(t_0,\ldots,t_{2k})x
 ~=~ v^*Hv
    \end{split}
\end{equation*}
for all $(t_0,\ldots,t_{2k}) \in \R^{2k+1}$. This implies that 
\[
\inf_{t_0,\ldots,t_{2k} \in \R}  \psi(t_0,\ldots,t_{2k}) \geq v^*Hv,
\]
for all $v \in \C^{n}$ such that $\|v\|=1$ and $v^TS_jv=0$ for each $j=0,\ldots,k$. This yields that 
\[
 \inf_{t_0,\ldots,t_{2k} \in \R}  \psi(t_0,\ldots,t_{2k})  \geq m_{hs_1\ldots s_k}(H,S_0,\ldots,S_k).
\]
\end{proof} 

The next result considers the case when $\psi$ attains its minimum at a point $(\hat t_0,\ldots,\hat t_{2k})$ where it is a simple eigenvalue of $F(\hat t_0,\ldots,\hat t_{2k})$. This will be useful in showing that the upper bound in Lemma~\ref{lem:bound} is equal to the value of $m_{hs_0,\ldots,s_k}$.

\begin{lemma}\label{lem:vector}
Let the function $\psi:\R^{2k+1}\longmapsto \R$ be as defined in Lemma~\ref{lem:existence} for matrices $H \in {\rm Herm}(n)$ and $S_0,  \ldots, S_k\in {\rm Sym}(n)$. Suppose that $\psi$ attains a global minimum at $(\hat t_0,\ldots,\hat t_{2k})\neq(0,\ldots,0)$ and define $\widehat \lambda_2:=\psi(\hat t_0,\ldots,\hat t_{2k})$. If $\widehat \lambda_2$ is a simple eigenvalue of $F(\hat t_0,\ldots,\hat t_{2k})$, where 
$F(t_0,\ldots,t_{2k})$ is as defined in~Lemma~\ref{lem:existence}, then there exists a vector $v \in \C^n\setminus \{0\}$ such that
 \[
\frac{v^*Hv}{v^*v}=\widehat \lambda_2, \quad \text{and}\quad v^TS_jv=0~\,\text{for\,all}~j=0,\ldots,k,
\]
and 
\begin{equation*}
 m_{hs_0\ldots s_k}(H,S_0,\ldots,S_k)  = \widehat \lambda_2.
 \end{equation*}

\end{lemma}
\begin{proof}
From~\eqref{eq:defF}, we have that 
\begin{equation}\label{eq:vector1}
F(t_0,\ldots,t_{2k})=\begin{bmatrix}
H & \overline{f(t_0,\ldots,t_{2k})}\\
f(t_0,\ldots,t_{2k})& \overline{H}\end{bmatrix}
=G+t_0H_0+\cdots+t_{2k}H_{2k},
\end{equation}
where $f(t_0,\ldots,t_{2k})$ and  the matrices $G$, $H_0,\ldots,H_{2k}$ are defined by~\eqref{eq:defmat1}--\eqref{eq:defmat3}. The function $\psi$ has a global minimum $\widehat \lambda_2$ at a point $(\hat t_0,\ldots,\hat t_{2k})\in \R^{2k+1}\setminus \{0\}$ and $\widehat \lambda_2$ is a simple eigenvalue of the matrix $F(\hat t_0,\ldots,\hat t_{2k})$. Then in view of Lemma~\ref{dlemma}, $\psi$ is partially differentiable at $(\hat t_0,\ldots,\hat t_{2k})$ and 
\begin{equation}\label{eq:vector2}
\frac{\partial \psi}{\partial t_j} (\hat t_0,\ldots,\hat t_{2k}) = v_0^*H_jv_0,\quad j=0,\ldots, 2k,
\end{equation}
where $v_0$ is a unit eigenvector of $F(\hat t_0,\ldots,\hat t_{2k})$ corresponding to the eigenvalue $\widehat \lambda_2$. Since $\widehat \lambda_2$ is the global minimum of $\psi$, we get 
\[
\frac{\partial \psi}{\partial t_j} (\hat t_0,\ldots,\hat t_{2k})=0,\quad \text{for~all}~j=0,\ldots, 2k,
\]
which implies from~\eqref{eq:vector2} that 
\begin{equation}\label{eq:vector3}
v_0^*H_jv_0=0,\quad \text{for~all}~j=0,\ldots, 2k.
\end{equation}
Next, let $G_0=H-\widehat \lambda_2 I_n$ and let $v_0=\mat{c}x\\ \overline y\rix$ for some $x, y \in \C^n$. Then from~\eqref{eq:vector1}, we have
\[
\begin{bmatrix}
H & \overline{f(\hat t_0,\ldots,\hat t_{2k})}\\
f(\hat t_0,\ldots,\hat t_{2k}) & \overline{H}\end{bmatrix}
\mat{c}x\\ \overline y\rix=\widehat \lambda_2 \mat{c}x\\ \overline y\rix
\]
if and only if 
\begin{equation}\label{eq:vector4}
    \begin{split}
        G_0x & =-\overline{(f(\hat t_0,\ldots,\hat t_{2k})y}\\
        G_0y & =- \overline{f(\hat t_0,\ldots,\hat t_{2k})x}.
    \end{split}
\end{equation}
Also from~\eqref{eq:vector3}, for each $j=0,\ldots,k-1$, we have that
\begin{eqnarray*}
v_0^*H_{2j}v_0=0\quad &\Longleftrightarrow &\quad 
\mat{c}x\\ \overline y\rix^* 
\mat{cc}0 & \overline{S_j} \\ S_j &0 \rix
\mat{c}x\\ \overline y\rix =0 \quad \Longleftrightarrow \quad 
\real{x^TS_jy}=0\\
v_0^*H_{2j+1}v_0=0\quad &\Longleftrightarrow& \quad 
\mat{c}x\\ \overline y\rix^* 
\mat{cc}0 & \overline{iS_j} \\ iS_j &0 \rix
\mat{c}x\\ \overline y\rix =0 \quad \Longleftrightarrow \quad 
\imag{x^TS_jy}=0.
\end{eqnarray*}
This implies that
$ x^TS_jy =0$ for all $j=0,\ldots, k-1$.
Similarly 
\begin{equation}\label{eq:vector5}
v_0^*H_{2k}v_0=0\quad \Longleftrightarrow \quad \real{x^TS_ky}=0.
\end{equation}
 Using this in~\eqref{eq:vector4}, we get 
$
x^*G_0x=-\overline{x^Tf(\hat t_0,\ldots,\hat t_{2k})y} = -\overline{\hat t_{2k} x^T S_k y} $ 
which implies that $x^TS_k y \in \R$, since $G_0$ is Hermitian and thus~\eqref{eq:vector5} implies that $x^TS_ky=0$. Thus far we have that
\[
x^TS_jy=0 \quad \text{for~all}~j=0,\ldots, k,\quad \text{and} \quad x^*G_0x=0~\Longrightarrow ~ x^*Hx=\hat \lambda_2 x^*x.
\]
Therefore, the proof is over if we show that $y=\alpha x$ for some $\alpha \in \C\setminus \{0\}$. For this note that in view of ~\eqref{eq:vector4}, it is easy to check that $\mat{c}x \\ \overline{y} \rix$ is an eigenvector of $F(\hat t_0,\ldots,\hat t_{2k})$ corresponding to the eigenvalue $\widehat \lambda_2$ if and only if $\mat{c}y \\ \overline{x} \rix$ is  an eigenvector of $F(\hat t_0,\ldots,\hat t_{2k})$ corresponding to the eigenvalue $\widehat \lambda_2$. Since $\widehat \lambda_2$ is a simple eigenvalue of 
$F(\hat t_0,\ldots,\hat t_{2k})$, the eigenspace corresponding to eigenvalue $\widehat \lambda_2$ is of dimension one. This implies that 
$\mat{c}x \\ \overline{y} \rix=\alpha \mat{c}y \\ \overline{x} \rix$ for some $\alpha \in \C$ with $|\alpha|=1$, since $\|v_0\|=1$, and thus we have $y=\alpha x$. This completes the proof. 
 \end{proof}

As a summary of this section, we have the following result that computes 
$m_{hs_1\ldots s_k}$ and $\widetilde m_{hs_1\ldots s_k}$.
 
\begin{theorem}\label{maintheorem}
Let $H \in {\rm Herm}(n)$ and $S_0, \ldots,S_k\in {\rm Sym}(n)$,
let $\psi: \R^{2k+1}\mapsto \R$ be defined by $(t_0,\ldots,t_{2k})\mapsto \lambda_2(F(t_0,\ldots,t_{2k}))$, and let $\widetilde \psi: \R^{2n+1}\mapsto \R$ be defined by $(t_0,\ldots,t_{2k})\mapsto \lambda_{2n-1}(F(t_0,\ldots,t_{2k}))$, where $F$ is defined by~\eqref{eq:defF}. 
Then,
\begin{equation}\label{eq:summary1}
    \begin{split}
        m_{hs_0\ldots s_k}(H,S_0,\ldots,S_k) & \leq \inf_{t_0,\ldots,t_{2k}\in\R}\psi(t_0,\ldots,t_{2k})\\
        \tilde{m}_{hs_0\ldots s_k}(H,S_1,\ldots,S_k) & \geq \sup_{t_0,\ldots,t_{2k}\in\R}\tilde{\psi}(t_0,\ldots,t_{2k}).
    \end{split}
\end{equation}
Moreover, the following statements hold.
\begin{enumerate}
    \item If ${\rm rank}(f(t_0,\ldots,t_{2k}))\geq 2$ for all $(t_0,\ldots,t_{2k}))\in \R^{2k+1}$, where $f$ is defined by~\eqref{eq:defmat1},
then both infimum and supremum in~\eqref{eq:summary1} are attained in the region $t_0^2+\cdots + t_{2k}^2\leq \beta^2$ where $\beta=\frac{\lambda_{\max}(H)-\lambda_{\min}(H)}{c}$ and $c=\min\big\{\lambda_2(t_0{H_0}+\cdots + t_{2k}{H_{2k}}) :~  t_0^2+\cdots + t_{2k}^2=1\big\}$, 
where $H_0,\ldots,H_{2k}$ are defined by~\eqref{eq:defmat2}--\eqref{eq:defmat3}.
\item Suppose that the infimum (supremum) in~\eqref{eq:summary1} is attained at $(\hat t_0,\ldots,\hat t_{2k})$.  If $\psi(\hat t_0,\ldots,\hat t_{2k})$ $\left(\widetilde{\psi}(\hat t_0,\ldots,\hat t_{2k})\right)$ is a simple eigenvalue of $F(\hat t_0,\ldots,\hat t_{2k})$, then 
equality holds in~\eqref{eq:summary1}.
\end{enumerate}
\end{theorem}

\begin{proof}
The statements about $m_{hs_0\ldots s_k}$ and $\psi$ follow immediately from Lemmas~\ref{lem:existence},~\ref{lem:vector}, and~\ref{lem:bound}.  The corresponding results for  $\widetilde m_{hs_0\ldots s_k}$ and $\widetilde \psi$ are then follow because of the facts that 
\[
\widetilde m_{hs_0\ldots s_k}(H,S_0,\ldots,S_k)=-m_{hs_0\ldots s_k}(-H,-S_0,\ldots,-S_k),
\]
and $\widetilde \psi(t_0,\ldots,t_{2k})=- \psi(-t_0,\ldots,-t_{2k})$ for all $(t_0,\ldots,t_{2k}) \in \R^{2k+1}$.
\end{proof}

We close the section with the following remark on problem~\eqref{eq:cenprob} with two symmetric constraints. 

\begin{remark}{\rm
When we consider $m_{hs_0s_1}$, i.e., problem~\eqref{eq:cenprob} with two symmetric constraints, the rank condition in Theorem~\ref{maintheorem} for the existence of the global minimum of 
$\psi$ can be modified to an explicit condition on ranks of $S_0$ and $S_1$
which is independent of the parameters $(t_0,t_1,t_2)\in \R^3$. More precisely, if we assume $\text{rank}(S_0)\geq 2$, $\text{rank}(S_1)\geq 2$, and 
$\text{rank}(\mat{cc}S_0 & S_1\rix)+\text{rank}(\mat{c}S_0\\S_1\rix) \geq 
\text{rank}(S_0)+\text{rank}(S_1)+2$, then by Lemma~\ref{changecondition}  $\text{rank}(S_0+S_1) \geq 2$. This implies that for any $(t_0,t_1,t_2)\in \R^3  \setminus \{0\}$, $\text{rank}((t_0+it_1)S_0+t_2S_1) \geq 2$. Indeed, when 
$t_0+it_1= 0$ or $t_2 = 0$, this is immediate, since $\text{rank}(S_0)\geq 2$ and $\text{rank}(S_1)\geq 2$. Now suppose $t_0+it_1 \neq 0$ and $t_2 \neq 0$, then consider $S'=(t_0+it_1)S_0$ and $S''=t_2S_1$. This implies that 
\[
\text{rank}(\mat{cc}S_0 & S_1\rix)=\text{rank}(\mat{cc}S' & S''\rix),\quad \text{and}\quad \text{rank}(\mat{c}S_0\\S_1\rix)=\text{rank}(\mat{c}S'\\S''\rix),
\]
and from Lemma~\ref{changecondition}, $\text{rank}((t_0+it_1)S_0+t_2S_1) \geq 2$.
}
\end{remark}

\section{Applications: Computing eigenvalue backward errors for structured polynomials}\label{sec:applications}
\label{sec:sber}

Let $P(z)=\sum_{j=0}^mz^jA_j$ be a regular matrix polynomial with $(A_0,\ldots,A_m)\in \mathbb S \subseteq (\C^{n,n})^{m+1}$ and let $\lambda \in \C$. Then the structured eigenvalue backward error of $\lambda$ with respect to $P$ and $\mathbb S$, is defined by
\begin{equation}\label{eq:defsber}
\eta^{\mathbb S}(P,\lambda):= \inf\left\{ \left(\sum_{j=0}^m {\|\Delta_j\|}^2\right)^{\frac{1}{2}}:~(\Delta_0,\ldots,\Delta_m)\in \mathbb S,\, 
\text{det}\left(\sum_{j=0}^m\lambda^j(A_j-\Delta_j)\right)=0
 \right\}.
\end{equation}
When $\mathbb S=(\C^{n,n})^{m+1}$, then $\eta(P,\lambda):=\eta^{\mathbb S}(P,\lambda)$ is called the unstructured eigenvalue backward error. The unstructured backward error is well known in~\cite[Proposition 4.6]{AhmA09} and given by 
\begin{equation}\label{eq:unsber}
\eta(P,\lambda)=\frac{\sigma_{\min}\left(P(\lambda)\right)}{\sqrt{1+|\lambda|^2 +\cdots + |\lambda^{2m}|^2}}.
\end{equation}

As mentioned earlier, the structured eigenvalue backward errors have been studied, and computable formulas were obtained for various structured matrix polynomials in~\cite{BorKMS14,BorKMS15}. It was shown in~\cite{BorKMS15,Sha16} that for T-palindromic, T-antipalindromic, skew-symmetric, T-even, and T-odd polynomials, computation of structured eigenvalue backward error is reduced  to the problem of solving $m_{hs_0\ldots s_k}$, where $k$ depends on the degree of $P$ and matrices $H,S_0,\ldots,S_k$ depend on the structure of $P$. For these structures, when the polynomial is of degree one (i.e.,  when $m=1$), $\eta^{\mathbb S}(P,\lambda)$
was obtained by solving $m_{hs_0}$ using Theorem~\ref{thm:mikaresult}. However, for higher degree polynomials (i.e., when $m>1$), the structured eigenvalue backward errors were not known in~\cite{BorKMS14,BorKMS15,Sha16}, due to unavailability 
 of a computable formula for the reduced problem in $m_{hs_0,\ldots,s_k}$. Therefore, in this section, we compute structured eigenvalue backward errors for T-palindromic, T-antipalindromic, skew-symmetric, T-even, and T-odd polynomials by solving the reduced problem in $m_{hs_0 \ldots s_k}$ using Theorem~\ref{maintheorem}.

\subsection{T-palindromic and T-antipalindromic polynomials}
\label{sec:Palsber}

Consider a matrix polynomial $P(z)=\sum_{j=0}^mz^jA_j$. Then $P$ is called 
T-palindromic if $A_j^T=A_{m-j}$ and it is called T-antipalindromic if $A_j^T=-A_{m-j}$, for each $j=0,\ldots,m$. Let ${\rm pal_T}$ (${\rm antipal_T}$) denote the subset of $(\C^{n,n})^{m+1}$ such that $(\Delta_0,\ldots,\Delta_m)\in {\rm pal_T}$ ($  {\rm antipal_T}$) implies $\Delta_j^T=\Delta_{m-j}$ ($\Delta_j^T=-\Delta_{m-j}$) for each $j=0,\ldots,m$, and the corresponding eigenvalue backward errors from~\eqref{eq:defsber} are denoted by $\eta^{{\rm pal_T}}(P,\lambda)$ when $\mathbb S= {\rm pal_T}$ and $\eta^{{\rm antipal_T}}(P,\lambda)$ when $\mathbb S={\rm antipal_T}$. For both T-palindromic and T-antipalindromic polynomials, if $\lambda \in \C\setminus \{0\}$ such that $\lambda =\pm 1$, then there is no difference between structured eigenvalue backward error and unstructured eigenvalue backward error~\cite[Theorem 5.3.1]{Adh08}. However, this is not the case if $\lambda \neq \pm 1$.
 It was noticed in~\cite[Theorem 3.2]{BorKMS15} that computing $\eta^{{\rm pal_T}}(P,\lambda)$ or $\eta^{{\rm antipal_T}}(P,\lambda)$ reduced to a problem of the form~\eqref{eq:cenprob}. 
We state this reduction as a result in the form that will be useful for us. 

For this we define,
$\Lambda_m:=[1,\lambda,\ldots,\lambda^m]\in \C^{1\times(m+1)}$ for $\lambda \in \C$, $k:=\lfloor \frac{m-1}{2} \rfloor$ and for each $j=0,\ldots,k$,  
$\gamma_{j1}:=\sqrt{\frac{2}{1+|\lambda|^{m-2j}}},\; \gamma_{j2}:=\sqrt{\frac{2|\lambda|^{m-2j}}{1+|\lambda|^{m-2j}}}$. Also let  
\begin{equation}
    \Gamma:= \begin{cases} 
      {\rm diag}(\gamma_{01},\ldots,\gamma_{k1},\gamma_{k2},\ldots,\gamma_{02})\otimes I_n & \text{if m is odd} \\
      {\rm diag}(\gamma_{01},\ldots,\gamma_{k1},1,\gamma_{k2},\ldots,\gamma_{02})\otimes I_n & \text{if m is even}.
      \end{cases}
\end{equation}

\begin{theorem}{\rm \cite{BorKMS15}}\label{thm:palsbrbmks}
Let $\mathbb S=\{{\rm pal_T},{\rm antipal_T} \}$ and let $P(z)=\sum_{j=0}^mz^jA_j$ be a regular matrix polynomial with $(A_0,\ldots,A_m) \in \mathbb S$. Suppose that $\lambda \in \C \setminus \{0,\pm 1\}$ such that $M:=(P(\lambda))^{-1}$ exists and define $H:=\Gamma^{-1}((\Lambda^*_m\Lambda_m)\otimes(M^*M))\Gamma^{-1}$. 
\begin{itemize}
\item {\rm (when $\mathbb S={\rm pal_T} $)} Then 
\[
\eta^{{\rm pal_T}}(P,\lambda)
=\left( m_{hs_0,\ldots,s_k}(H,S_0,\ldots,S_k)\right )
^{-\frac{1}{2}},
\]
where for each $j=0,\ldots,k$
\[
S_j:=\Gamma^{-1}(C_j+C^T_j)\Gamma^{-1}  \quad \text{with} \quad  C_j=(\Lambda^T_me^T_{j+1})\otimes M^T - (e_{m-j+1}\Lambda_m)\otimes M,
\]
where $e_{j+1}$ denotes the $j^{th}$ standard basis vector of $\R^{m+1}$.
\item {\rm (when $\mathbb S={\rm antipal_T} $)} Then 
\[
\eta^{{\rm antipal_T}}(P,\lambda)
=\left( m_{h\tilde s_0,\ldots,\tilde s_k}(H,\tilde S_0,\ldots,\tilde S_k)\right )
^{-\frac{1}{2}}, 
\]
if $m$ is odd, and
\[
\eta^{{\rm antipal_T}}(P,\lambda)
=\left( m_{h\tilde s_0,\ldots,\tilde s_k}(H,\tilde S_0,\ldots,\tilde S_k,\tilde S_{\frac{m}{2}})\right )
^{-\frac{1}{2}}, 
\]
if $m$ is even, where for each $j=0,\ldots,k$
\[
\tilde S_j:=\Gamma^{-1}(\tilde C_j+\tilde C^T_j)\Gamma^{-1} \quad \text{and} \quad \tilde S_{\frac{m}{2}}:=\Gamma^{-1}(\tilde C_{\frac{m}{2}})\Gamma^{-1} 
\]
with
\[
\tilde C_j=(\Lambda^T_me^T_{j+1})\otimes M^T + (e_{m-j+1}\Lambda_m)\otimes M \quad \text{and}\quad \tilde C_{\frac{m}{2}}=(\Lambda^T_me^T_{\frac{m}{2}+1})\otimes M^T + 
(e_{\frac{m}{2}+1}\Lambda_m)\otimes M,
\]
where $e_{j+1}$ denotes the $j^{th}$ standard basis vector of $\R^{m+1}$.
\end{itemize}
\end{theorem}

If we consider the function $\psi : \R^{2k+1}\mapsto \R$ as defined in Theorem~\ref{maintheorem} for the symmetric matrices $S_j$'s or $\tilde S_j$'s of Theorem~\ref{thm:palsbrbmks}, it will always has a global minimum. This is because the symmetric matrices satisfy the rank condition in Theorem~\ref{maintheorem} required for the global minimum of $\psi$. This is shown in the following result for symmetric matrices $S_0,\ldots,S_k$. However, we note that  an analogous result can be obtained for matrices $\tilde S_0,\ldots,\tilde S_k$ using similar arguments. 
This will be used in formulating a computable formula for $\eta^{{\rm pal_T}}(P,\lambda) $ and $\eta^{{\rm antipal_T}}(P,\lambda) $.

\begin{lemma}\label{lem:rankpal}
Let the matrices $ S_0,\ldots, S_k$ be as defined in Theorem~\ref{thm:palsbrbmks}. Then for any $(t_0,\ldots,t_{2k})\in \R^{2k+1}\setminus \{0\}$,
\[
{\rm rank}\left((t_0+it_1)S_0 + (t_2+it_3) S_1+\cdots+(t_{2k-2}+it_{2k-1}) S_{k-1}+ t_{2k} S_k\right) =2n .
\]
\end{lemma}
\begin{proof} In view of Theorem~\ref{thm:palsbrbmks}, for any $(t_0,\ldots,t_{2k}) \in \R^{2k+1 }\setminus \{0\}$, we have
\begin{eqnarray}\label{eq:prof11}
&(t_0+it_1)S_0  +\cdots +(t_{2k-2}+it_{2k-1})S_1+t_{2k}S_k =\\
&\Gamma^{-1}\Big(\Lambda_m^T\big((t_0+it_1)e_1^T+\cdots + (t_{2k-2}+it_{2k-1}) e_k^T +t_{2k}e_{k+1}^T- t_{2k}e_{m-k+1}^T - (t_{2k-2}+it_{2k-1}) e_{m-k+2}^T\nonumber \\
&-\cdots 
-(t_0+it_1)e_{m+1}^T\big)\otimes M^T 
+ \big((t_0+it_1)e_1+\cdots + (t_{2k-2}+it_{2k-1}) e_k +t_{2k}e_{k+1}- t_{2k}e_{m-k+1} \nonumber \\
& - (t_{2k-2}+it_{2k-1}) e_{m-k+2}-\cdots - 
(t_0+it_1)e_{m+1}\big) \Lambda_m \otimes M\Big)\Gamma^{-1} \nonumber \\
& =\Gamma^{-1}\left(\Lambda_m^T\Omega_m \otimes M^T+\Omega_m^T\Lambda_m \otimes M
\right)\Gamma^{-1},
\end{eqnarray}
where $\Omega_m:=(t_0+it_1)e_1^T+\cdots + (t_{2k-2}+it_{2k-1}) e_k^T +t_{2k}e_{k+1}^T- t_{2k}e_{m-k+1}^T - (t_{2k-2}+it_{2k-1}) e_{m-k+2}^T-\cdots -(t_0+it_1)e_{m+1}^T$ and $M$, $\Lambda_m$, and $\Gamma$ are as defined in Theorem~\ref{thm:palsbrbmks}. In view of~\eqref{eq:prof11}, 
we will prove that ${\rm rank}((t_0+it_1)S_0  +\cdots +(t_{2k-2}+it_{2k-1})S_1+t_{2k}S_k) =2n$ by showing that 
${\rm rank}(\Lambda_m^T\Omega_m \otimes M^T+\Omega_m^T\Lambda_m \otimes M)=2n$. For this, let  $J_1=\Lambda_m^T\Omega_m \otimes M^T$ and $J_2=\Omega_m^T\Lambda_m \otimes M$. Then $J_2=J_1^T$. Note that the first block row of $J_1$ is $\Omega_m \otimes M^T$ and all other rows are scalar multiple of this row. This implies that 
\[
{\rm rank}\left(J_1\right)= {\rm rank}\left(\Omega_m \otimes M^T\right) =n,
\]
since $\Omega_m \neq 0$ (as not all $t_j$'s zero) and $M$ is invertible. 
Thus also ${\rm rank}(J_2)={\rm rank}(J_1^T)=n$.
The proof is over if we show that ${\rm rank}\left(\mat{c}J_1 \\ J_2\rix\right)=2n$, because this will imply that 
\[
{\rm rank}\left(\mat{cc}J_1 & J_2\rix\right)= {\rm rank}\left(\mat{c}J_1^T \\  J_2^T\rix\right)=
{\rm rank}\left(\mat{c}J_2 \\ J_1\rix\right)={\rm rank}\left(\mat{c}J_1 \\ J_2\rix\right)=2n.
\]
Now from Lemma~\ref{changecondition},
\begin{eqnarray}\label{eq:proof_pallemma1}
{\rm rank}(J_1+J_2) &\geq& {\rm rank}\left(\mat{cc}J_1 & J_2\rix\right) +
{\rm rank}\left(\mat{c}J_1 \\ J_2\rix\right) -
{\rm rank}(J_1) -{\rm rank}(J_2) \nonumber \\ 
&=&  2n+2n-n-n=2n 
\end{eqnarray}
and also 
\begin{equation}\label{eq:proof_pallemma2}
{\rm rank}(J_1+J_2) \leq {\rm rank}(J_1)+{\rm rank}(J_2) =n+n=2n.
\end{equation}
Thus~\eqref{eq:proof_pallemma1} and~\eqref{eq:proof_pallemma2} will then imply that 
\begin{equation*}
{\rm rank}\left( \Lambda_m^T\Omega_m \otimes M^T+\Omega_m^T\Lambda_m \otimes M \right) ={\rm rank}(J_1+J_2)=2n.
\end{equation*}

Therefore in rest of the proof we will show that  ${\rm rank}\left(\mat{c}J_1 \\ J_2\rix\right) =2n$. Note that not all block rows of $J_2$ are zero, since not all
$t_j$'s are zero, w.l.o.g. lets assume that $t_0+it_1 \neq 0$. Consider the first block row of $J_2$ which is $(t_0+it_1)\Lambda_m \otimes M$, then 
\[
{\rm rank}((t_0+it_1)\Lambda_m \otimes M) = n= {\rm rank} (J_2),
\]
since $M$ is invertible. Also the first $n$ rows of $J_1$ are given by $\Omega_m \otimes M^T$, and we have 
\[
{\rm rank}(\Omega_m \otimes M^T) = n= {\rm rank} (J_1).
\]
Thus in order to show that ${\rm rank}\left(\mat{c}J_1 \\ J_2\rix\right) =2n$, it is enough to show that 
\[
{\rm rank}\left(\mat{c}\Omega_m \otimes M^T \\ (t_0+it_1)\Lambda_m \otimes M \rix \right)=2n.
\]
 For this, let $A_j$ and $B_j$ denote the $j^{th}$
row of $\Omega_m \otimes M^T $ and  $(t_0+it_1)\Lambda_m \otimes M$, respectively, and let $r_j$ and $c_j$ denote respectively the $j^{th}$ row and $j^{th}$ column of $M$. Then 
\begin{equation}\label{eq:proof12}
A_j=\Omega_m \otimes c_j^T\quad \text{and} \quad B_j= (t_0+it_1)\Lambda_m \otimes r_j.
\end{equation}
We show that for each $p=1,\ldots,n$, the set $D_p=\left\{B_p,A_1,\ldots,A_n\right\}$ is linearly independent. On contrary, suppose that $D_p$ is linearly dependent for some $p \in \{1,\ldots,n\}$, then there exists scalars $\alpha_{jp} \in \C$ such that 
\[
B_p=\sum_{j=1}^{n} \alpha_{jp}A_j = \sum_{j=1}^n \alpha_{jp} (\Omega_m \otimes c_j^T) = \sum_{j=1}^n (\alpha_{jp}\Omega_m \otimes c_j^T).
\]
That means from~\eqref{eq:proof12}
\[
(t_0+ it_1) \Lambda_m \times r_p = \sum_{j=1}^n (\alpha_{jp}\Omega_m \otimes c_j^T),
\]
or equivalently, 
\begin{eqnarray}\label{eq:proof_pallemma3}
&\mat{cccc}(t_0+ it_1)r_p & (t_0+ it_1) \lambda r_p&\cdots& (t_0+ it_1)\lambda^m r_p\rix \nonumber \\
&=\mat{cccccc} \sum_{j=1}^n \alpha_{jp}(t_0+it_1)c_j^T & \cdots &\sum_{j=1}^n \alpha_{jp}(t_{2k})c_j^T & \sum_{j=1}^n \alpha_{jp}(t_{2k})c_j^T& \cdots &
 \sum_{j=1}^n \alpha_{jp}(t_0+it_1)c_j^T
\rix \nonumber \\
\end{eqnarray}
On comparing the first term in both sides, we get 
\begin{eqnarray*}
(t_0+ it_1)r_p=  (t_0+it_1) \sum_{j=1}^n \alpha_{jp}c_j^T,
\end{eqnarray*}
which implies that 
\begin{equation}\label{eq:proof_pallemma4}
r_p= \sum_{j=1}^n \alpha_{jp}c_j^T.
\end{equation}
Substituting~\eqref{eq:proof_pallemma4} in~\eqref{eq:proof_pallemma3}, we get
\begin{equation}\label{eq:proof_pallemma5}
(t_0+ it_1) \Lambda_m \times r_p= \Omega_m \otimes r_p \quad
\Longrightarrow \quad  (t_0+ it_1) \Lambda_m = \Omega_m.
\end{equation}
Again comparing both sides in~\eqref{eq:proof_pallemma5} and using the fact that $(t_0+it_1)\neq 0$, we get either $\lambda=1$ or $\lambda=-1$, which is a contradiction, since $\lambda \in \C\setminus \{0,\pm 1\}$. Hence
$\{B_p,A_1,\ldots,A_n\}$ is linearly independent for each $p=1,\ldots,n$. Also 
since $\{B_1,\ldots,B_n\}$ is linearly independent, we have that 
$\{B_1,\ldots,B_n,A_1,\ldots,A_n\}$ is linearly independent. This implies that 
${\rm rank}\left(\mat{c}J_1 \\ J_2\rix\right) =2n$ which completes the proof. 
\end{proof}

An application of Theorem~\ref{maintheorem} in Theorem~\ref{thm:palsbrbmks} yields a computable formula for the structured eigenvalue backward errors of palindromic polynomials. More precisely, we have the following results. 

\begin{theorem}\label{thm:sberpal}
Let $P(z)=\sum_{j=0}^mz^j A_j$ be a regular matrix polynomial with $(A_0,\ldots,A_m) \in {\rm pal_T}$. Suppose that $\lambda \in \C\setminus \{0,\pm 1\}$ be such that $M:=(P(\lambda))^{-1}$ exists. Define $k=\lfloor{\frac{m-1}{2}}\rfloor$ and let 
$\psi: \R^{2k+1}\mapsto \R$ be defined by $(t_0,\ldots,t_{2k})\mapsto \lambda_2(F(t_0,\ldots,t_{2k}))$, where $F(t_0,\ldots,t_{2k})$  is as defined in Theorem~\ref{maintheorem} for matrices $H,S_0,\ldots,S_k$ of Theorem~\ref{thm:palsbrbmks}.
 Then 
\begin{itemize}
\item The function $\psi$ has a global minimum
\[
\widehat \lambda_2 := \min_{(t_0,\ldots,t_{2k})\in \R^{2k+1}} \psi (t_0,\ldots,t_{2k}) \quad \text{and}\quad 
\eta^{{\rm pal_T}}(P,\lambda) \geq \frac{1}{\sqrt{\widehat \lambda_2}}.
\]
\item Moreover, if the minimum $\widehat \lambda_2$ of $\psi$  is attained 
at $(\hat t_0,\ldots,\hat t_{2k}) \in \R^{2k+1}$ and is a simple eigenvalue of $F(\hat t_0,\ldots,\hat t_{2k}) $, then 
\[
\eta^{{\rm pal_T}}(P,\lambda) = \frac{1}{\sqrt{\widehat \lambda_2}}.
\]
\end{itemize}
\end{theorem}
\begin{proof}
In view of Lemma~\ref{lem:rankpal}, the matrices $ S_0,\ldots, S_k$
satisfy that for any $(t_0,\ldots,t_{2k}) \in \R^{2k+1}\setminus \{0\}$,
\[
\text{rank}\left((t_0+it_1) S_0 + (t_2+it_3)S_1+\cdots+(t_{2p-2}+it_{2p-1}) S_{k-1}+ t_{2p} S_k\right) =2n \geq 2,
\]
since $n \geq 1$. Thus from Theorem~\ref{maintheorem}, $\psi$ attains its global minimum. Now the proof is immediate by using Theorem~\ref{maintheorem} in Theorem~\ref{thm:palsbrbmks}.
\end{proof}

\begin{theorem}\label{thm:sberantipal}
Let $P(z)=\sum_{j=0}^mz^j A_j$ be a regular matrix polynomial with $(A_0,\ldots,A_m) \in {\rm antipal_T}$. Suppose that $\lambda \in \C\setminus \{0,\pm 1\}$ be such that $M:=(P(\lambda))^{-1}$ exists. Define $k=\lfloor{\frac{m-1}{2}}\rfloor$ and let $p=k$ if $m$ is odd and $p=k+1$ if $m$ is even.  
Let 
$\psi: \R^{2p+1}\mapsto \R$ be defined by $(t_0,\ldots,t_{2p})\mapsto \lambda_2(F(t_0,\ldots,t_{2p}))$, where $F(t_0,\ldots,t_{2p})$  is as defined in Theorem~\ref{maintheorem} for matrices $H,\tilde S_0,\ldots,\tilde S_p$ of Theorem~\ref{thm:palsbrbmks}.
 Then 
\begin{itemize}
\item The function $\psi$ has a global minimum
\[
\widehat \lambda_2 := \min_{(t_0,\ldots,t_{2p})\in \R^{2p+1}} \psi (t_0,\ldots,t_{2p}) \quad \text{and}\quad 
\eta^{{\rm antipal_T}}(P,\lambda) \geq \frac{1}{\sqrt{\widehat \lambda_2}}.
\]
\item Moreover, if the minimum $\widehat \lambda_2$ of $\psi$  is attained 
at $(\hat t_0,\ldots,\hat t_{2p}) \in \R^{2p+1}$ and is a simple eigenvalue of $F(\hat t_0,\ldots,\hat t_{2p}) $, then 
\[
\eta^{{\rm antipal_T}}(P,\lambda) = \frac{1}{\sqrt{\widehat \lambda_2}}.
\]
\end{itemize}
\end{theorem}
\begin{proof}
The proof follows by arguments similar to those of Theorem~\ref{thm:sberpal}.
\end{proof}

\subsection{T-even and T-odd polynomials}\label{sec:altsber}

A matrix  polynomial $P(z)=\sum_{j=0}^mz^jA_j$ is called T-even (T-odd) if $A_j^T=(-1)^jA_j$ $\left(A_j^T=(-1)^{j+1}A_j\right)$, for $j=0,\ldots,m$.
Let $\mathbb S_e$($\mathbb S_o$) denote the subset of $(\C^{n,n})^{m+1}$ such that $(\Delta_0,\ldots,\Delta_m) \in 
\mathbb S_e (\mathbb S_o)$ implies that $\Delta_j^T=(-1)^j\Delta_j$ $\left(\Delta_{j}^T=(-1)^{j+1}\Delta_{j}\right)$, for $j=0,\ldots,m$, and let the corresponding eigenvalue backward errors from~\eqref{eq:defsber} be denoted by $\eta^{{\rm even_T}}(P,\lambda)$ when $\mathbb S= \mathbb S_e$ and $\eta^{{\rm odd_T}}(P,\lambda)$ when $\mathbb S=\mathbb S_o$.  Then from~\cite[Theorem 3.4.1]{Sha16} computing $\eta^{{\rm even_T}}(P,\lambda)$ and $\eta^{{\rm odd_T}}(P,\lambda)$ reduced to a problem of the form~\eqref{eq:cenprob}. We state this as a result in the following.

\begin{theorem}{\rm \cite{Sha16}}\label{thm:reduction_alt}
Let $P(z)=\sum_{j=0}^mz^jA_j$ be a regular matrix polynomial, where 
$(A_0,\ldots,A_m)\in \mathbb S \subseteq (\C^{n,n})^{m+1}$. Let $\lambda \in \C\setminus \{0\}$ such that $M:=(P(\lambda))^{-1}$ exists. Define 
$\Lambda_m:=[1\, \lambda\, \ldots \, \lambda^m] \in \C^{1 \times (m+1)}$ and $H:=\Lambda_m^*\Lambda_m \otimes M^*M$. 
\begin{itemize}
\item {\rm (when $\mathbb S=\mathbb S_e$)} Let $k=\left\lfloor \frac{m-1}{2} \right\rfloor$ and for each $j=0,\ldots,k$, let
\[
C_{ej}:=(\Lambda^T_me^T_{2j+2})\otimes M^T,\quad \text{and}\quad  S_{ej}:=C_{ej}+C_{ej}^T,
\]
where $e_j$ denotes the $j^{th}$ standard basis vector of $\R^{m+1}$.
Then 
\[
\eta^{{\rm even_T}}(P,\lambda)=(m_{hs_{e0}\ldots s_{ek}}(H,S_{e0},\ldots,S_{ek}))^{-\frac{1}{2}}.
\]
\item {\rm (when $\mathbb S=\mathbb S_o$)} Let $k=\left \lfloor \frac{m}{2} \right\rfloor$ and for each $j=0,\ldots,k$, let
\[
C_{oj}:=(\Lambda^T_me^T_{2j+1})\otimes M^T \quad \text{and}\quad  S_{oj}:=C_{oj}+C_{oj}^T.
\]
Then 
\[
\eta^{{\rm odd_T}}(P,\lambda)=(m_{hs_{o0}\ldots s_{ok}}(H,S_{o0},\ldots,S_{ok}))^{-\frac{1}{2}}.
\]

\end{itemize}
\end{theorem}

Next, we obtain a result similar to Lemma~\ref{lem:rankpal} for matrices
$S_{e0},\ldots, S_{ek}$ of Theorem~\ref{thm:reduction_alt}. This will be used in computing $\eta^{{\rm even_T}}(P,\lambda)$ and $\eta^{{\rm odd_T}}(P,\lambda)$. Note that an analogous result for $S_{o0},\ldots, S_{ok}$ follows similarly. 

\begin{lemma}\label{lem:rank_alt}
Let the matrices $S_{e0},\ldots, S_{ek}$ be as defined in Theorem~\ref{thm:reduction_alt}. Then for any $(t_0,\ldots,t_{2k})\in \R^{2k+1}\setminus \{0\}$,
\[
{\rm rank}\left((t_0+it_1)S_{e0} + (t_2+it_3) S_{e1}+\cdots+(t_{2k-2}+it_{2k-1}) S_{ek-1}+ t_{2k} S_{ek}\right) \geq 2n .
\]
\end{lemma}
\begin{proof}
For any $(t_0,\ldots,t_{2k})\in \R^{2k+1}\setminus \{0\}$, we have
\begin{equation}\label{eq:prooflemalt1}
\left((t_0+it_1)S_{e0} + (t_2+it_3) S_{e1}+\cdots+(t_{2k-2}+it_{2k-1}) S_{ek-1}+ t_{2k} S_{ek}\right) =\Lambda_m^T \Omega_m \otimes M^T + \Omega_m^T \Lambda_m \otimes M,
\end{equation}
where $\Lambda_m:=[1\, \lambda\, \ldots \, \lambda^m] \in \C^{1 \times (m+1)}$ and
\begin{equation*}
\Omega_m:= \begin{cases} 
\left [0 ~t_0+it_1~ 0~t_2+it_3~0~\cdots 0~t_{2k-2}+it_{2k-1}~0~t_{2k}\right ]      
       & \text{if m is odd} \\
      \left [0 ~t_0+it_1~ 0~t_2+it_3~0~\cdots ~0~t_{2k-2}+it_{2k-1}~0~t_{2k}~0 \right ]   & \text{if m is even}.
      \end{cases}
\end{equation*}
Set 
\[
J_1:=\Lambda_m^T\Omega_m \otimes M^T=
\begin{bmatrix}
0 & (t_0+it_1)M^T & \cdots & 0 & t_{2k}M^T\\
0 & \lambda (t_0+it_1)M^T & \cdots & 0 & \lambda t_{2k}M^T\\
\vdots & \vdots &  & \vdots & \vdots \\
0 & \lambda^m (t_0+it_1)M^T & \cdots & 0 & \lambda^mt_{2k} M^T
\end{bmatrix},
\]
when $m$ is odd, and 
\[
J_1:=\Lambda_m^T\Omega_m \otimes M^T=
\begin{bmatrix}
0 & (t_0+it_1)M^T & \cdots & 0 & t_{2k}M^T&0\\
0 & \lambda (t_0+it_1)M^T & \cdots & 0 & \lambda t_{2k}M^T & 0\\
\vdots & \vdots &  & \vdots & \vdots& \vdots \\
0 & \lambda^m (t_0+it_1)M^T & \cdots & 0 & \lambda^mt_{2k} M^T & 0
\end{bmatrix},
\]
when $m$ is even, 
then 
\begin{equation}\label{eq:prooflemalt2}
\Lambda_m^T \Omega_m \otimes M^T + \Omega_m^T \Lambda_m \otimes M=J_1+J_1^T.
\end{equation}
Observe that, the first block row of $J_1$ is nonzero as not all $t_j$'s are zero. Suppose that $j^{th}$ block entry of the first row is nonzero, then the first and $j^{th}$ block columns of $J_1+J_1^T$ are linearly independent. This shows that $\text{rank}(J_1+J_1^T) \geq 2n$. Thus the result follows from~\eqref{eq:prooflemalt1} and~\eqref{eq:prooflemalt2}.
\end{proof}

As an application of Theorem~\ref{maintheorem} in Theorem~\ref{thm:reduction_alt} yields a computable formula for the structured eigenavalue backward error of T-even and T-odd polynomials. For future reference, we state the results separately for $\eta^{{\rm even_T}}(P,\lambda)$ and $\eta^{{\rm odd_T}}(P,\lambda)$.

\begin{theorem}\label{thm:formulaTeven}
Let $P(z)=\sum_{j=0}^mz^jA_j$ be a regular T-even matrix polynomial, where 
$(A_0,\ldots,A_m)\in \mathbb S_e $. Let $\lambda \in \C \setminus \{0\}$
be such that $M:=(P(\lambda))^{-1}$ exists. Define 
 $k=\left \lfloor{\frac{m-1}{2}}\right \rfloor$ and let 
$\psi: \R^{2k+1}\mapsto \R$ be defined by $(t_0,\ldots,t_{2k})\mapsto \lambda_2(F(t_0,\ldots,t_{2k}))$, where $F(t_0,\ldots,t_{2k})$  is as defined in Theorem~\ref{maintheorem} for matrices $H,S_{e0},\ldots,S_{ek}$ of Theorem~\ref{thm:reduction_alt}. Then 
\begin{itemize}
\item The function $\psi$ has a global minimum
\[
\widehat \lambda_2 := \min_{(t_0,\ldots,t_{2k})\in \R^{2k+1}} \psi (t_0,\ldots,t_{2k}) \quad \text{and}\quad 
\eta^{{\rm even_T}}(P,\lambda) \geq \frac{1}{\sqrt{\widehat \lambda_2}}.
\]
\item Moreover, if the minimum $\widehat \lambda_2$ of $\psi$  is attained 
at $(\hat t_0,\ldots,\hat t_{2k}) \in \R^{2k+1}$ and is a simple eigenvalue of $F(\hat t_0,\ldots,\hat t_{2k}) $, then 
\[
\eta^{{\rm even_T}}(P,\lambda) = \frac{1}{\sqrt{\widehat \lambda_2}}.
\]
\end{itemize}
\end{theorem}
\begin{proof}
In view of Lemma~\ref{lem:rank_alt}, the matrices $ S_{e0},\ldots, S_{ek}$
satisfy that for any $(t_0,\ldots,t_{2k}) \in \R^{2k+1}\setminus \{0\}$,
\[
\text{rank}\left((t_0+it_1) S_{e0} + (t_2+it_3)S_{e1}+\cdots+(t_{2p-2}+it_{2p-1}) S_{ek-1}+ t_{2p} S_{ek}\right) =2n \geq 2,
\]
since $n \geq 1$. Thus from Theorem~\ref{maintheorem}, $\psi$ attains its global minimum. Now the proof is immediate using Theorem~\eqref{maintheorem} in Theorem~\ref{thm:reduction_alt}.
\end{proof}

Analogous to Theorem~\ref{thm:formulaTeven}, in the following we obtain 
a formula for the structured eigenvalue backward error of T-odd polynomials.

\begin{theorem}\label{thm:formulaTodd}
Let $P(z)=\sum_{j=0}^mz^jA_j$ be a regular T-odd matrix polynomial, where 
$(A_0,\ldots,A_m)\in \mathbb S_o $. Let $\lambda \in \C \setminus \{0\}$
be such that $M:=(P(\lambda))^{-1}$ exists. Define 
 $k:=\left \lfloor{\frac{m}{2}}\right \rfloor$ and let 
$\psi: \R^{2k+1}\mapsto \R$ be defined by $(t_0,\ldots,t_{2k})\mapsto \lambda_2(F(t_0,\ldots,t_{2k}))$, where $F(t_0,\ldots,t_{2k})$  is as defined in Theorem~\ref{maintheorem} for matrices $H,S_{o0},\ldots,S_{ok}$ of Theorem~\ref{thm:reduction_alt}. Then 
\begin{itemize}
\item The function $\psi$ has a global minimum
\[
\widehat \lambda_2 := \min_{(t_0,\ldots,t_{2k})\in \R^{2k+1}} \psi (t_0,\ldots,t_{2k}) \quad \text{and}\quad 
\eta^{{\rm odd_T}}(P,\lambda) \geq \frac{1}{\sqrt{\widehat \lambda_2}}.
\]
\item Moreover, if the minimum $\widehat \lambda_2$ of $\psi$  is attained 
at $(\hat t_0,\ldots,\hat t_{2k}) \in \R^{2k+1}$ and is a simple eigenvalue of $F(\hat t_0,\ldots,\hat t_{2k}) $, then 
\[
\eta^{{\rm odd_T}}(P,\lambda) = \frac{1}{\sqrt{\widehat \lambda_2}}.
\]
\end{itemize}
\end{theorem}
\begin{proof}
A result similar to Lemma~\ref{lem:rank_alt} can be obtained for matrices
$S_{o0},\ldots,S_{ok}$. Thus the proof follows on the lines of Theorem~\ref{thm:formulaTeven}.
\end{proof}

\subsection{skew-symmetric matrix polynomials}\label{sec:skewsber}

Next, we consider skew-symmetric matrix polynomial $P(z)=\sum_{j=0}^m z^jA_j$ with $A_j^T=-A_j$ for each $j=0,\ldots,m$.  Let $ SSym(n)$ denote the set of all skew-symmetric matrices of size $n \times n$ and let the skew-symmetric backward error (when $\mathbb S=(SSym(n))^{m+1}$) from~\eqref{eq:defsber} be denoted by $\eta^{{\rm skew_T}}(P,\lambda)$.
Then according to~\cite[Theorem 5.2.20]{Sha16}, computing $\eta^{{\rm skew_T}}(P,\lambda)$ is equivalent to solving a problem of the form~\eqref{eq:cenprob}. We first state this reduction from~\cite{Sha16} and then obtain a formula for 
$\eta^{{\rm skew_T}}(P,\lambda$ using Theoerm~\ref{maintheorem}.

\begin{theorem}{\rm \cite{Sha16}}\label{thm:reduction_skew}
Let$P(z)=\sum_{j=0}^m z^jA_j$ be a regular skew-symmetric matrix polynomial, where $A_j \in SSym(n)$ for each $j=0,\ldots,m$.
Suppose that $\lambda \in \C\setminus \{0\}$ such that  $M=(P(\lambda))^{-1}$ exists. Define 
$\Lambda_m:=[1\, \lambda\, \ldots \, \lambda^m] \in \C^{1 \times (m+1)}$,  $H:=\Lambda_m^*\Lambda_m \otimes M^*M$, and for each $j=0,\ldots,m$ define 
\[
C_{j}:=(\Lambda^T_m e^T_{j+1})\otimes M^T \quad \text{and} \quad  S_{j}:=C_{j}+C_{j}^T,
\]
where $e_{j+1}$ is the $j^{th}$ standard basis vector of $\R^{m+1}$.
Then, 
\[
\eta^{{\rm skew_T}}(P,\lambda)=(m_{hs_0\ldots s_m}(H,S_0,\ldots,S_m))^{-\frac{1}{2}}.
\]
\end{theorem}

As an application of Theorem~\ref{maintheorem} in Theorem~\ref{thm:reduction_skew} yields a computable formula for $\eta^{{\rm skew_T}}(P,\lambda)$ and is given in the following result.

\begin{theorem}\label{thm:formulaSkew}
Let $P(z)=\sum_{j=0}^mz^jA_j$ be a regular skew-symmetric matrix polynomial, where 
$A_j \in SSym(n)$ for each $j=0,\ldots,m$. Let $\lambda \in \C \setminus \{0\}$
be such that $M:=(P(\lambda))^{-1}$ exists. Let 
$\psi: \R^{2m+1}\mapsto \R$ be defined by $(t_0,\ldots,t_{2m})\mapsto \lambda_2(F(t_0,\ldots,t_{2m}))$, where $F(t_0,\ldots,t_{2m})$  is as defined in Theorem~\ref{maintheorem} for matrices $H,S_{0},\ldots,S_{m}$ of Theorem~\ref{thm:reduction_skew}. Then 
\begin{itemize}
\item  If ${\rm rank}(f(t_0,\ldots,t_{2m}))\geq 2$ for all $(t_0,\ldots,t_{2m})\in \R^{2m+1}$, where $f$ is defined by~\eqref{eq:defmat1}
for matrices $S_{0},\ldots,S_{m}$, then the function $\psi$ has a global minimum
\[
\widehat \lambda_2 := \min_{(t_0,\ldots,t_{2m})\in \R^{2m+1}} \psi (t_0,\ldots,t_{2m}) \quad \text{and}\quad 
\eta^{{\rm skew_T}}(P,\lambda) \geq \frac{1}{\sqrt{\widehat \lambda_2}}.
\]
\item Moreover, if the minimum $\widehat \lambda_2$ of $\psi$  is attained 
at $(\hat t_0,\ldots,\hat t_{2m}) \in \R^{2m+1}$ and is a simple eigenvalue of $F(\hat t_0,\ldots,\hat t_{2m}) $, then 
\[
\eta^{{\rm skew_T}}(P,\lambda) = \frac{1}{\sqrt{\widehat \lambda_2}}.
\]
\end{itemize}
\end{theorem}
%
%

\begin{remark}{\rm
We note that unlike Lemmas~\ref{lem:rankpal} and~\ref{lem:rank_alt}, an analogous result  for the symmetric matrices 
$S_0,\ldots,S_m$ of Theorem~\ref{thm:reduction_skew} need not be true, that means, for any $(t_0,\ldots,t_m) \in \R^{2m+1}\setminus \{0\}$,
 $\text{rank}((t_0+it_1)S_0+\cdots+(t_{2m-2}+it_{2m-1})S_{m-1}+t_{2m}S_m) \geq 2n$ is not necessarily true. For example, in the pencil case (when $m=1$), the symmetric matrices are given as 
 \[
 S_0=\Lambda_m^Te_1^T \otimes M^T + e_1\Lambda_m \otimes M
 =\mat{cc}M+M^T & \lambda M \\ \lambda M^T & 0
 \rix
 \]
 and 
  \[
 S_1=\Lambda_m^Te_2^T \otimes M^T + e_2\Lambda_m \otimes M
 =\mat{cc}0 &  M \\  M^T & \lambda (M+M^T)
 \rix.
 \]
Thus if we take $t_0,t_1 \in \R$ and set $t_2=\lambda(t_0+it_1)$, then 
$\text{rank}((t_0+it_1)S_0+t_2S_1) \leq n$.
}
\end{remark}

\section{Numerical Experiments}
\label{sec:numerical}

In this section, we illustrate the proposed methods for computing $m_{h s_0 \ldots s_m}$. To compute the formula in Theorem~\ref{maintheorem}, we used the 
{\it GlobalSearch} in Matlab Version No. 9.5.0 (R2018b) to solve the associated optimization problem.

\subsection{Comparison with the upper bound in~{\rm \cite{Sha16}}}

Note that for given matrices $H \in {\rm Herm}(n)$ and $S_0,\ldots,S_k \in {\rm Sym}(n)$, an upper bound for $m_{hs_0 \ldots s_m}$ was obtained in~\cite[Theorem 3.2.2]{Sha16} in terms of $2^{k+1}$th largest eigenvalue of a recursively defined parameter depending Hermitian matrix of size $2^{(k+1)n} \times 2^{(k+1)n}$. This upper bound was conjectured to be equal to $m_{h S_0 \ldots S_m}$ 
in~\cite[Conjecture 3.2.8]{Sha16} if the eigenvalue at the optimal is simple. 
We compare our result (Theorem~\ref{maintheorem}) of computing $m_{h s_0 \ldots s_m}$ with the bound in~\cite[Theorem 3.2.2]{Sha16}.

We generate the matrices $H \in {\rm Herm}(n)$ and $S_0,\ldots,S_1 \in {\rm Sym}(n)$ for different values of $n$ randomly using the {\it randn} and  {\it rand} commands in Matlab. We chose these matrices such that the optimal eigenvalues in Theorem~\ref{maintheorem} and in~\cite[Theorem 3.2.2]{Sha16}  are simple. This will imply the exact formula for $m_{h s_0 s_1}$ by Theorem~\ref{maintheorem}. Also, this will satisfy the hypothesis in~\cite[Conjecture 3.2.8]{Sha16}.
The computed results are depicted in Table~\ref{tab:compare1}, where the right and left side charts  show the results, respectively, when the matrices are generated using {\it rand} and {\it randn} commands in Matlab. We make two observations: 

\begin{itemize}
\item Our estimation to $m_{hs_0s_1}$ in Theorem~\ref{maintheorem}  is better than the one obtained in~\cite[Theorem 3.2.2]{Sha16};
\item Since the values in the second column ($m_{hs_0s_1}$ according to Theorem~\ref{maintheorem}) is smaller than the values in the third column (upper bound in~\cite[Theorem 3.2.2]{Sha16}), we conclude that the conjecture in~\cite{Sha16} is not true, and  it is in general only an upper bound to $m_{hs_0s_1}$. 
\end{itemize}

\begin{table}[ht]
        \centering
\begin{tabular}{cc}
        \begin{tabular}{|c|c|c|}
            \hline
             n & $m_{hs_0s_1}$ & upper bound \\ 
             &Thoerem~\ref{maintheorem} & \cite[Theorem 3.2.2]{Sha16}\\   \hline
            3 & 8.9239 & 8.9421 \\  \hline
            4 & 11.4218 & 11.4374 \\ \hline
            5 & 17.4434  & 17.8928 \\ \hline
            6 & 24.6630 & 24.6895 \\ \hline
            7 & 23.9694 & 23.9913 \\  \hline
        \end{tabular} &
       \begin{tabular}{|c|c|c|}
            \hline
             n & $m_{hs_0s_1}$ & upper bound \\ 
             &Thoerem~\ref{maintheorem} & \cite[Theorem 3.2.2]{Sha16}\\ \hline
            3 & 0.9346 & 0.9347 \\ \hline
            4 & 2.1144 & 2.1860 \\ \hline
            5 & 2.5708 & 2.6423 \\ \hline
            6 & 3.5511 & 3.5555 \\ \hline
            7 & 3.3251 & 3.3402 \\  \hline
        \end{tabular}
\end{tabular}
        \caption{Estimation to $m_{hs_0 s_1}$ from Thoerem~\ref{maintheorem} and \cite[Theorem 3.2.2]{Sha16}} \label{tab:compare1}
\end{table}

\subsection{Estimation to structured eigenvalue backward errors}

In view of Sections~\ref{sec:Palsber}-\ref{sec:skewsber}, we can
estimate the structured eigenvalue backward errors of T-palindromic, T-antipalindromic, T-even, T-odd, and skew-symmetric polynomials by computing $m_{h s_0 \ldots s_k}$. We  include examples only for T-palindromic and T-even polynomials by noting similar results for other structures also. 

\begin{example}{\rm (T-palilndromic polynomial)
$P(z)=A_0+zA_1+z^2A_1^T+z^3A_0^T$ is a $3 \times 3$ T-palindromic polynomial of degree $3$, where the entries of the coefficient matrices are generated using the {\it randn} command in Matlab. 

Table~\ref{table:2} (left side) records  eigenvalue backward errors of $P(z)$ for random $\lambda$. The second column stands for the unstructured eigenvalue backward error $\eta(P,\lambda)$ (from~\eqref{eq:unsber}) and the third column records the structured eigenvalue backward error $\eta^{{\rm pal_T}}(P,\lambda)$ obtained in Theorem~\ref{thm:sberpal}. Note that, we found the exact value of $\eta^{{\rm pal_T}}(P,\lambda)$ up to four decimal points for these examples, since the eigenvalue at the optimal is simple (see Theorem~\ref{thm:sberpal}). This may be due to the fact that the corresponding matrices $H$ and $S_0,S_1$ are special with some specific structures. This shows that, as expected, the eigenvalue backward errors with respect to structure-preserving and unstructured perturbations are significantly different. 

Next, we chose $\lambda$ values such that they approach to the eigenvalue    $-0.8493-0.5910i$ of  $P(z)$. The corresponding results are depicted in Table~\ref{table:2} (right side). Again the unstructured and the structured eigenvalue backward errors are different. Also, as expected, both unstructured and structured eigenvalue backward errors tend towards zero with $\lambda$ approaching  the eigenvalue $-0.8493-0.5910i$. 

\begin{table}[h!]
\centering
\begin{tabular}{cc}
\begin{tabular}{|c|c|c|}
\hline
 $\lambda$ & $\eta(P,\lambda)$ & $\eta^{{\rm pal_T}}(P,\lambda)$ \\ 
 &\eqref{eq:unsber} & Theorem~\ref{thm:sberpal} \\
\hline
0.629+1.023i & 1.2115 & 1.5684   \\ 
\hline
0.296+0.732i & 1.0175 & 1.3515  \\ 
\hline
1.091-0.283i & 2.4592 & 2.5712   \\
\hline
0.498-0.941i & 1.1356 & 1.4728   \\
\hline
-1.503-0.137i & 0.9617 & 1.1007  \\
\hline
-0.723+1.364i & 0.4613 & 0.5689   \\
\hline
\end{tabular}&
\begin{tabular}{ | c|c|c|}
\hline
$\lambda$ & $\eta(P,\lambda)$ & $\eta^{palT}(P,\lambda)$  \\ 
 &\eqref{eq:unsber} & Theorem~\ref{thm:sberpal} \\
\hline
-0.815-0.565i & 0.0615 & 0.0764   \\ 
\hline
-0.820-0.570i & 0.0514 & 0.0638    \\ 
\hline
-0.825-0.575i & 0.0414 & 0.0514  \\
\hline
-0.830-0.580i & 0.0315 & 0.0392   \\
\hline
-0.835-0.585i & 0.0220 & 0.0273    \\
\hline
-0.840-0.590i & 0.0132 & 0.0164    \\
\hline
\end{tabular}
\end{tabular}
\caption{Structured and unstructured eigenvalue backward errors for T-palindromic polynomial $P(z)$}
\label{table:2}
\end{table}

}
\end{example}

\begin{example}{\rm (T-even polynomial)
 $L(z)=A_0+zA_1+z^2A_2+z^3A_3$ is a $4 \times 4$ T-even polynomial of degree $3$, where the entries of the coefficient matrices are generated using the {\it randn} command in Matlab. The structured eigenvalue backward error, $\eta^{{\rm even_T}}(P,\lambda$, is obtained using Theorem~\ref{thm:formulaTeven}, and the unstructured backward error is computed using~\eqref{eq:unsber}.
 We record these values  in Table~\ref{table:3} (left side) for randomly generated $\lambda$. The two eigenvalue backward errors $\eta(P,\lambda)$ and $\eta^{{\rm even_T}}(P,\lambda)$ are significantly different. This shows that eigenvalues are less sensitive under structure-preserving perturbations. Table~\ref{table:3} (right side) records the result when $\lambda$ approaches to the eigenvalue $-0.9544+1.8341i$ of $L(z)$. In this case, both unstructured and structured backward errors tend towards zero.

\begin{table}[h!]
\centering
\begin{tabular}{cc}
\begin{tabular}{ |c|c|c|}
\hline
\centering $\lambda$ & $\eta(P,\lambda)$ & $\eta^{{\rm even_T}}(P,\lambda)$  \\ 
&\eqref{eq:unsber} &Theorem~\ref{thm:formulaTeven}\\
\hline
0.502-0.0123i & 0.76750 & 0.83996    \\ 
\hline
0.552-0.0173i & 0.78366 & 0.86632   \\ 
\hline
0.602-0.0223i & 0.79410 & 0.88302    \\
\hline
0.652-0.0273i & 0.79742 & 0.88720    \\
\hline
0.702-0.0323i & 0.79261 & 0.87838    \\
\hline
0.752-0.0373i & 0.77926 & 0.85792    \\
\hline
\end{tabular}
&
\begin{tabular}{ | c|c|c|}
\hline
\centering $\lambda$ & $\eta(P,\lambda)$ & $\eta^{{\rm even_T}}(P,\lambda)$  \\ 
&\eqref{eq:unsber} &Theorem~\ref{thm:formulaTeven}\\
\hline
-0.152+1.708i & 0.09741 & 0.10343  \\ 
\hline
-0.302+1.858i & 0.07952 & 0.08614   \\ 
\hline
-0.452+2.008i & 0.06361 & 0.06366   \\
\hline
-0.602+2.158i & 0.03849 & 0.04169   \\
\hline
-0.752+2.308i & 0.02038 & 0.02215   \\
\hline
-0.902+2.458i & 0.00457 & 0.00499   \\
\hline
\end{tabular}
\end{tabular}
\caption{Structured and unstructured eigenvalue backward errors for T-even polynomial $L(z)$}
\label{table:3}
\end{table}
}
\end{example}

\section{Conclusion}

In this paper, we have derived an estimation of the supermum of the Rayleigh quotient of a Hermitian matrix with respect to constraints involving symmetric matrices. Our estimate is in the form of minimizing the second largest eigenvalue of a parameter depending Hermitian matrix. 
We have then applied these results to compute the structured eigenvalue backward errors of matrix polynomials with T-palindromic, T-antipalindromic, T-even, T-odd, and skew-symmetric structures. 

\bibliographystyle{siam}
\bibliography{bibliostable}

\end{document}